\documentclass[12pt,epsfig,amsfonts]{amsart}
\usepackage{amsmath,amsthm,amssymb,mathrsfs,amscd,epsfig,color}
\usepackage{url}
\usepackage{graphicx}
\usepackage{mathrsfs}
\usepackage{ulem}
\usepackage{fancybox}
\usepackage{pb-diagram} 
\usepackage[all]{xy}
\usepackage{comment}
\usepackage{mathtools}
\usepackage{centernot}
\usepackage{placeins} 
\usepackage{url}

\makeatletter

\newcommand{\confrac}[2]{%
  \frac{\displaystyle{%
    \strut\hfill{#1}\hfill\;\vrule}}%
      {\displaystyle{%
       \strut\vrule\;\hfill{#2}\hfill}}}%

\makeatletter

    \newcommand\contFrac{\@ifstar{\@contFracStar}{\@contFracNoStar}}

   \def\singleContFrac#1#2{%
        \begin{array}{@{}c@{}}%
            \multicolumn{1}{c|}{#1}%
            \\%
            \hline%
           \multicolumn{1}{|c}{#2}%
        \end{array}%
   }

    \def\@contFracNoStar#1{%
        \mathchoice{
            \@contFracNoStarDisplay@#1//\@nil%
        }{
            \@contFracNoStarInline@#1//\@nil%
        }{
            \@contFracNoStarInline@#1//\@nil%
        }{
            \@contFracNoStarInline@#1//\@nil%
        }%
    }

    \def\@contFracNoStarDisplay@#1//#2\@nil{%
        \@ifmtarg{#2}{%
            #1%
        }{%
            #1+\cfrac{1}{\@contFracNoStarDisplay@#2\@nil}%
        }%
    }

        \def\@contFracNoStarInline@#1//#2\@nil{%
            \@ifmtarg{#2}{%
                #1%
            }{%
                #1 \@@contFracNoStarInline@@#2\@nil%
            }%
        }
        \def\@@contFracNoStarInline@@#1//#2\@nil{%
            \@ifmtarg{#2}{%
                + \singleContFrac{1}{#1}%
            }{%
                + \singleContFrac{1}{#1} \@@contFracNoStarInline@@#2\@nil%
            }%
        }

    \def\@contFracStar#1{%
        \mathchoice{
            \@contFracStarDisplay@#1////\@nil%
        }{
            \@contFracStarInline@#1//\@nil%
        }{
            \@contFracStarInline@#1//\@nil%
        }{
            \@contFracStarInline@#1//\@nil%
        }%
    }

    \def\@contFracStarDisplay@#1//#2//#3\@nil{%
        \@ifmtarg{#2}{%
            #1%
        }{%
            #1 + \cfrac{#2}{\@contFracStarDisplay@#3\@nil}%
        }%
    }

        \def\@contFracStarInline@#1//#2\@nil{%
            \@ifmtarg{#2}{%
                #1%
            }{%
                #1 \@@contFracStarInline@@#2\@nil%
            }%
        }
        \def\@@contFracStarInline@@#1//#2//#3\@nil{%
            \@ifmtarg{#3}{%
                + \singleContFrac{#1}{#2}%
            }{%
                + \singleContFrac{#1}{#2} \@@contFracStarInline@@#3\@nil%
            }%
        }
\makeatother

       \numberwithin{equation}{section}
\theoremstyle{plain}
\newtheorem{thm}{Theorem}[section]
\newtheorem{lem}[thm]{Lemma}
\newtheorem{cor}[thm]{Corollary}

\newtheorem{prop}[thm]{Proposition}
\theoremstyle{definition}

\newtheorem{rem}[thm]{Remark}

\makeatletter
\@namedef{subjclassname@2020}{%
  \textup{2020} Mathematics Subject Classification}
\makeatother
\title[Hausdorff dimension of intersections ]
{ Hausdorff dimension of intersections between\\ the Jarn\'ik sets and Diophantine fractals
}
\author{Hiroki Takahasi}

\address{Keio Institute of Pure and Applied Sciences (KiPAS),  Department of Mathematics, Keio University, Yokohama,223-8522, JAPAN}  \email{hiroki@math.keio.ac.jp}

\begin{document}

\begin{abstract} 
The irrationality exponent of a real number measures how well that number can be approximated by rationals.
Real numbers with irrationality exponent strictly greater than $2$ are transcendental numbers, and 
form a set with rich fractal structure.
We show that this set intersects 
the limit set of any parabolic iterated function system arising from the backward continued fraction in a set of full Hausdorff dimension. As a corollary, we show that the set of irrationals whose irrationality exponents are strictly bigger than $2$ and whose backward continued fraction expansions have bounded partial quotients is of Hausdorff dimension $1$.
This is a sharp contrast to the fact that there exists no irrational whose irrationality exponent is strictly greater than $2$ and whose regular continued fraction expansion has bounded partial quotients.

\end{abstract}

\maketitle



\section{Introduction}\label{intro}
 The {\it irrationality exponent}
 $\mu(x)$ of a real number $x$ measures
 how well $x$ can be approximated by rational numbers. It is the supremum of the set of $\mu\in\mathbb R$ such that 
 the inequality
 \[\left|x-\frac{p}{q}\right|<\frac{1}{q^\mu}\]
 holds for 
 an infinite number of 
 $(p,q)\in\mathbb Z\times\mathbb N$ 
 such that $|p|$ and $q$ are coprime. 
 Rational numbers have irrationality exponent equal to $1$. Irrational numbers have irrationality exponent bigger than or equal to $2$. 
 From the work of Khintchine \cite{Khi24},  
 almost all irrational numbers in the sense of the Lebesgue measure
 have irrationality exponent equal to $2$ (see \cite[Section~1]{Sch71}). 
 By the Thue-Siegel-Roth theorem \cite{R55}, real numbers with irrationality exponent strictly greater than $2$ are transcendental numbers. 

  The set of numbers with irrationality exponent strictly greater than $2$ is of Lebesgue measure $0$, yet has rich fractal structure.
  For each $\alpha>2$, the set 
\[G(\alpha)=\{x\in\mathbb R\setminus\mathbb Q\colon\mu(x)\geq\alpha\}\]
is called {\it the Jarn\'ik set}.
 Jarn\'ik \cite{Jar29} and Besicovitch \cite{Bes34} independently showed that  
 $G(\alpha)$ is of Hausdorff dimension $2/\alpha$
 for any $\alpha>2$.
G\"uting \cite{Gut63} showed that the set $\{x\in\mathbb R\setminus\mathbb Q\colon\mu(x)=\alpha\}$ is of Hausdorff dimension $2/\alpha$ for any $\alpha>2$, thereby strengthening the result of Jarn\'ik and Besicovitch.
Beresnevich et al. \cite{BDV01} showed that the Hausdorff $2/\alpha$-measure of this set is infinite.
Bugeaud \cite{Bug03}  extended these results
to general approximation order functions in the sense of Khintchine \cite{Khi24}.
Hill and Velani \cite{HV98} proved an analogue of the Jarn\'ik-Besicovitch theorem for geometrically finite Kleinian groups. 
Falconer \cite{Fal94} 
showed that
$G(\alpha)$
has the large intersection property for any $\alpha\in(2,\infty)$.

We further investigate the structure of the Jar\'ik sets by analyzing
their intersections with other fractal sets.
Bugeaud \cite[Theorem~2]{Bug08} showed that 
for any $\alpha\geq2$, the middle third Cantor set in $[0,1]$ contains an irrational with irrationality exponent $\alpha$, answering the question of
Mahler \cite{Mah84}.
Becher et al. \cite[Theorem~1]{BRS} showed that for any $\alpha>2$ and any $b\in[0,2/\alpha]$, there exists a Cantor-like set such that with respect to the uniform probability measure on it, almost every number 
has irrationality exponent equal to $\alpha$.
In this paper we consider intersections of the Jarn\'ik sets and fractal sets arising from continued fractions.

Each irrational $x\in\mathbb R$
has the {\it regular}  continued fraction (RCF) expansion
\begin{equation}\label{RCF}\begin{split}x&=a_0(x)+\confrac{1 }{a_{1}(x)} + \confrac{1 }{a_{2}(x)}+ \confrac{1 }{a_{3}(x)}  +\cdots
\end{split}
\end{equation}
where $a_0(x)=\lfloor x\rfloor$ and 
$a_n(x)\geq1$, $n\geq1$ are integers called {\it partial quotients}. 
Each irrational $x\in\mathbb R$ also has the {\it backward} (aka minus or negative) continued fraction (BCF) expansion
\begin{equation}\label{BCF}\begin{split}x&=b_0(x)-\confrac{1 }{b_{1}(x)} - \confrac{1 }{b_{2}(x)} -\confrac{1 }{b_{3}(x)} -\cdots
,\end{split}\end{equation}
where $b_0(x)=\lfloor x\rfloor+1$ and 
$b_n(x)\geq2$, $n\geq1$ are integers.
Sets of irrationals defined by restrictions  
on their partial quotients become fractal sets.
For a finite set $\mathcal A\subset \mathbb N$ with $\#\mathcal A\geq2$, define
\[E_\mathcal A=\{x\in(0,1)\setminus\mathbb Q\colon a_n(x)\in \mathcal A\text{ for all }n\geq1\}.\]
Similarly, for a finite set $\mathcal B\subset \mathbb N_{\geq2}$
with $\#\mathcal B\geq2$
define
\[F_\mathcal B=\{x\in(0,1)\setminus\mathbb Q\colon b_n(x)\in \mathcal B\text{ for all }n\geq1\}.\]

Recall that $x\in\mathbb R\setminus\mathbb Q$ is {\it badly approximable} if there is a constant $C>0$ such that for any positive integer pair $(p,q)$,   \[\left|x-\frac{p}{q}\right|>\frac{C}{q^2}.\]  
Irrationals with bounded RCF partial quotients are precisely badly approximable numbers.
Since the irrationality exponent of a badly approximable number is $2$, we get
\begin{equation}\label{RCF}G(\alpha)\cap\{x\in\mathbb R\colon (a_n(x))_{n=1}^\infty\text{ is bounded}  \}=\emptyset\ \text{
for any }\alpha>2.\end{equation}
In particular,
for any finite set $\mathcal A\subset \mathbb N$ with $\#\mathcal A\geq 2$, $G(\alpha)\cap E_\mathcal A=\emptyset$ holds for any $\alpha>2$.

To move on to the BCF fractals, 
a key ingredient is the well-known transformation formulas between the RCF and BCF partial quotients (see Section~\ref{trans-sec}).
For any finite set
$\mathcal B\subset \mathbb N_{\geq2}$ with $\#\mathcal B\geq2$ and $2\notin \mathcal B$, the transformation formula from the BCF to RCF shows that any irrational in $F_\mathcal B$ is badly approximable. Hence, $G(\alpha)\cap F_\mathcal B=\emptyset$ holds for any $\alpha>2$.

The appearance of $2$ in the BCF partial quotients is necessary to obtain non-empty intersection.
From the result of Duverney and Shiokawa \cite[Theorem~3.3]{DuvShi24}, $G(\alpha)\cap  F_{\{2,3\}}$ is
non-empty for any $\alpha>2$. 
A close inspection into their proof shows that $\{2,3\}$ can be generalized to  
any finite set $\mathcal B\subset \mathbb N_{\geq2}$ with $\#\mathcal B\geq2$ and $2\in \mathcal B$,
 and that $G(\alpha)\cap  F_{\mathcal B}$ is uncountable for any $\alpha>2$. Our main result 
 is concerned with the Hausdorff dimension of these intersections. 
 Let $\dim_{\rm H}$ denote the Hausdorff dimension on $\mathbb R$.
\begin{thm}\label{thma}
For any finite set $\mathcal B\subset\mathbb N_{\geq2}$ with $\#\mathcal B\geq2$ and 
 $2\in \mathcal B$, we have
\[\lim_{\alpha\to2+0}\dim_{\rm H}(G(\alpha)\cap F_\mathcal B  )=\dim_{\rm H}F_\mathcal B.\]
\end{thm}
From Theorem~\ref{thma} we obtain the following sharp contrast to \eqref{RCF}.
\begin{thm}\label{thmb}
We have \[\lim_{\alpha\to2+0}\dim_{\rm H}(G(\alpha)\cap\{x\in\mathbb R\colon (b_n(x))_{n=1}^\infty\ \text{\rm is bounded}\})=1.\]
\end{thm}

As corollaries to Theorems~\ref{thma} and \ref{thmb} respectively we obtain the following statements.
\begin{cor}
For any finite set $\mathcal B\subset\mathbb N_{\geq2}$ with $\#\mathcal B\geq2$ and 
 $2\in \mathcal B$, we have
 \[\dim_{\rm H}\{x\in F_{\mathcal B}\colon\mu(x)>2\}=\dim_{\rm H}F_{\mathcal B}.\]
\end{cor}
\begin{cor}\label{corb}We have
\[\{x\in\mathbb R\colon \mu(x)>2 \text{ and } (b_n(x))_{n=1}^\infty\ \text{\rm is bounded}\}=1.\]\end{cor}


 

Our proof of Theorem~\ref{thma} relies on 
a dynamical systems approach and consists of three steps.
Given a finite set $\mathcal B\subset\mathbb N_{\geq2}$ as in the statement,
 we first construct a {\it seed set} that is independent of $\alpha$ and has Hausdorff dimension close to $\dim_{\rm H}F_{\mathcal B}$. We then construct a subset of $G(\alpha)\cap F_{\mathcal B}$ by modifying the seed set,
 for appropriately chosen $\alpha$ close to $2$. Finally we estimate the Hausdorff dimension of this set from below. 

It is well-known that the BCF is generated by an iteration of the map $T(x)=\frac{1}{1-x}-\lfloor\frac{1}{1-x}\rfloor$ on $[0,1)$ with a neutral fixed point $x=0$ introduced by R\'enyi \cite{R57}.
In particular, $F_{\mathcal B}$ may be viewed as a $T$-invariant set.
For our proofs of the main results, it is more useful to view
$F_{\mathcal B}$ as a subset of the limit set of an Iterated Function System (IFS) with neutral fixed point, called a {\it parabolic IFS} (see Section~\ref{IFS-sec}). The effect of the neutral fixed point must be taken into consideration in all our constructions. From this parabolic IFS we extract an `accelerated' IFS without neutral fixed point (see Proposition~\ref{katok}), and take its limit set as our seed set.

In the second step, 
we modify the seed set by inserting longer and longer blocks of $2$ into sequences of the BCF partial quotients to construct a subset of $G(\alpha)\cap F_{\mathcal B}$ 
with Hausdorff dimension close to $\dim_{\rm H}F_{\mathcal B}$.
According to the transformation formula from the BCF to RCF, 
a long block of $2$ in the sequence of BCF partial quotients is transformed into a single large RCF partial quotient.
We insert blocks of $2$ so that the resultant sequence of RCF partial quotients and partial denominators satisfy the criterion of Good \cite{G41} for
belonging to $G(\alpha)$ (see Lemma~\ref{identify}).

A crucial estimate 
in the last step is that on Euclidean diameters of fundamental intervals. We carefully choose  positions to insert blocks of $2$ and lengths of these blocks, so that the diameters of modified fundamental intervals do not decay substantially (see Lemma~\ref{compare} for details).

The rest of this paper consists of three sections.
In Section~2 we summarize basic properties and preliminary results on the RCF and BCF, and parabolic IFSs. In Section~3 we prove Theorems~\ref{thma} and \ref{thmb}. In Section~4 we comment on the validity of formula for irrationality exponents for the BCF.
\section{Preliminaries}
In Section~\ref{trans-sec} we summarize basic properties of the RCF and BCF. In Section~\ref{IFS-sec} we introduce Iterated Function Systems, and parabolic ones
following Mauldin and Urba\'nski \cite{MauUrb03}.
In Sections~\ref{ass}, \ref{large-sec}, \ref{bound-lem} we summarize preliminary results on parabolic IFSs that will be used in Section~3.
\subsection{Basic properties of the RCF and BCF}\label{trans-sec}
Let $x\in\mathbb R\setminus\mathbb Q$. Define two sequences $(p_n(x))_{n=-1}^\infty$, $(q_n(x))_{n=-1}^\infty$ of integers inductively by
\begin{equation}\label{pq}\begin{split}&p_{-1}(x)=1,\ p_0(x)=a_0(x),\ p_n(x)=a_n(x)p_{n-1}(x)+p_{n-2}(x),\\
&q_{-1}(x)=0,\ q_0(x)=1,\ q_n(x)=a_n(x)q_{n-1}(x)+q_{n-2}(x).\end{split}\end{equation}

\begin{lem}\label{qn-up}For all $x\in\mathbb R\setminus\mathbb Q$ and all $n\in\mathbb N$, we have
\[q_n(x)\leq\prod_{i=1}^{n}(a_{i}(x)+1).\]
\end{lem}
\begin{proof} By the definition \eqref{pq}, $(q_n(x))_{n=0}^\infty$ is increasing. Hence we have 
$q_n(x)\leq (a_n(x)+1)q_{n-1}(x)$ for  all $n\geq1$, which implies the desired inequality. \end{proof}

It is well-known \cite{Khi64} that for any
$x\in\mathbb R\setminus\mathbb Q$ and any $n\geq1$, \begin{equation}\label{appeq}\frac{1}{q_n(x)(q_n(x)+q_{n+1}(x))}<\left|x-\frac{p_n(x)}{q_n(x)}\right|<\frac{1}{q_n(x)q_{n+1}(x)}.\end{equation}
As in \cite[pp.204--205]{G41}, using \eqref{pq}, \eqref{appeq} and the fact that $(p_n(x)/q_n(x))_{n=1}^\infty$ provides the best rational approximation to $x$, one can show the following identification of the Jarn\'ik sets.
  \begin{lem}[{\cite[pp.204--205]{G41}}]\label{identify}For any $\alpha\in(2,\infty)$ we have
\[G(\alpha)=\{x\in\mathbb R\setminus\mathbb Q\colon a_{n+1}(x)> q_n^{\alpha-2}(x)\ \text{ for infinitely many }n\in\mathbb N\}.\]
\end{lem}

The RCF and BCF expansions of each
 $x\in\mathbb R\setminus\mathbb Q$ are transformed from one to the other as follows. In terms of $(a_n(x))_{n=1}^\infty$, the BCF expansion of $x$ is written as
\[x=\lfloor x\rfloor+1-\underbrace{\confrac{1 }{2} -\cdots- \confrac{1 }{2 }}_{a_1(x)-1 \text{ times}} -\confrac{1}{a_2(x)+2} -\underbrace{\confrac{1 }{2} -\cdots- \confrac{1 }{2 }}_{a_3(x)-1 \text{ times}} -\confrac{1}{a_4(x)+2}-\cdots
,\]
see 
\cite[Proposition~2,\ Remark~1]{DajKra00}. 
Conversely, 
put $n_0=0$ and define a sequence $(n_k)_{k=1}^\infty$ of positive integers inductively by
\[b_n(x)=2\ \text{ if }n_{k-1}<n<n_k,\text{ and } b_{n_k}(x)\geq3.\]
Then we have $a_0(x)=\lfloor x\rfloor$ and
\[a_{2k-1}(x)=n_{k}-n_{k-1},\ a_{2k}(x)=b_{n_k}(x)-2 \ \text{ for }k\geq1.\]
The point is that a long block of $2$ is transformed into a single large partial quotient.

\subsection{Parabolic Iterated Function System}\label{IFS-sec}
Let $X$ be a compact interval with positive Euclidean diameter.
Let $\mathcal I$ be a subset of $\mathbb N$ with $\#\mathcal I\geq2$, and 
let  $\phi_i\colon X\to X$ $(i\in \mathcal I)$ be $C^1$ maps, i.e., each $\phi_i$ can be extended to a $C^1$ map on an open set containing $X$.
The collection
$\Phi=\{\phi_i\}_{i\in \mathcal I}$ is called  {\it an Iterated Function System (IFS)} on $X$. It is called {\it an infinite (resp. finite)} IFS if $\mathcal I$ is an infinite (resp. finite) set.  We say an IFS $\Phi$ satisfies {\it the open set condition} if for all distinct indices $i,j\in \mathcal I,$ 
\[
\phi_i({\rm int}X)\cap \phi_j({\rm int}X)=\emptyset.
\]

Let $\Phi=\{\phi_i\}_{i\in \mathcal I}$
be an IFS on $X$. For $\omega=(\omega_1,\omega_2,\ldots)\in {\mathcal I}^{\mathbb N}$ and $n\in \mathbb{N}$, we set
\[\phi_{\omega_1\cdots \omega_n}=\phi_{\omega_1}\circ\cdots\circ\phi_{\omega_n}.\]
If the set $\bigcap_{n=1}^{\infty}\phi_{\omega_1\cdots\omega_{n}}(X)$ is a singleton for any 
$\omega\in {\mathcal I}^{\mathbb N}$,  we
 define {\it an address map} 
 $\Pi \colon {\mathcal I}^{\mathbb N} \to X$ by
\[\Pi(\omega)\in \bigcap_{n=1}^{\infty}\phi_{\omega_1\cdots\omega_n}(X),\]
and define {\it the limit set} 
 \[
\Lambda(\Phi)=\Pi( {\mathcal I}^{\mathbb N} ).
\]
Since the address map may not be injective, we introduce the set \[\Lambda'(\Phi)=\{x\in \Lambda(\Phi)\colon\#\Pi^{-1}(x)=1\}.\] 
Since $X$ is an interval, if the open set condition holds then
 $\Lambda(\Phi)\setminus \Lambda'(\Phi)$ is countable and of Hausdorff dimension zero.
For each $x\in \Lambda'(\Phi)$,
 there is a unique sequence $(\omega_n(x))_{n=1}^\infty\in {\mathcal I}^\mathbb N$  
satisfying $x=\Pi((\omega_n(x))_{n=1}^\infty)$.
Note that
\[
x = \lim_{n\to\infty} \phi_{\omega_1(x)} \circ \cdots \circ \phi_{\omega_n(x)} (y)\ \text{ for all }y\in X.
\]

An IFS $\Phi$ on $X$ is called {\it parabolic}
if the open set condition holds, and the following two conditions hold:
\begin{itemize}

\item[(A1)] (Non-uniform contraction) $|\phi'_i(x)|<1$ everywhere except for finitely many pairs $(i,x_i)$, $i\in \mathcal I$, for which $x_i$ is the unique fixed point of $\phi_i$ and $|\phi'_i(x_i)|=1$.
Such pairs and indices $i$ are called {\it parabolic}.

\item[(A2)] (Bounded distortion) 
There exists a constant $C\ge 1$ such that for all $\omega\in {\mathcal I}^{\mathbb N}$ and 
 $n\in\mathbb N_{\geq2}$ such that $\omega_{n}$ is not a parabolic index, or else $\omega_{n-1}\neq\omega_n$, 
\[
|\phi'_{\omega_1\cdots\omega_n}(x)|\le C|\phi'_{\omega_1\cdots\omega_n}(y)|\ \text{ for all }x,y\in X.
\]

\end{itemize}

 \begin{rem}
For any parabolic IFS on a compact interval, the address map and the limit set are well-defined, see
\cite[Proposition~3.1]{JT} and \cite[Lemma~2.2]{T25} for example.\end{rem}

\subsection{Saturation and mild distortion}\label{ass}
Let $\Phi=\{\phi_i\}_{i\in \mathcal I}$ be a parabolic IFS. For each $n\in\mathbb N$ let ${\mathcal I}^n$ denote the set of words from 
$\mathcal I$ with word length $n$. We endow the Cartesian product ${\mathcal I}^{\mathbb N}$ with the product topology of the discrete topology on $\mathcal I$.
Let $\mathcal M$ denote the set of shift invariant ergodic 
 Borel probability measures on ${\mathcal I}^{\mathbb N}$. For each $\nu\in\mathcal M$, 
 define 
 \[\chi(\nu)=-\int\log|\phi'_{\omega_1 }(\Pi(\omega))| d\nu(\omega)\in[0,\infty],\]
 where $\Pi\colon{\mathcal I}^{\mathbb N}\to\Lambda(\Phi)$ denotes the address map and $\omega_1$ denotes the first coordinate of $\omega\in{\mathcal I}^{\mathbb N}$.
  For a Borel probability measure $\mu$ on $[0,1]$, define \[\dim(\mu)=\inf\{\dim_{\rm H}A\colon A\subset[0,1],\ \mu(A)=1\}.\]
We say $\Phi$ is {\it saturated} if 
 \[\dim_{\rm H}\Lambda(\Phi)=\sup\left\{\dim(\nu\circ\Pi^{-1})\colon \nu\in\mathcal M,\ \chi(\nu)<\infty\right\}.\]
For each $n\in\mathbb N$
we define \[D_n(\Phi)=\sup_{\omega_1\cdots \omega_n\in {\mathcal I}^n }\max_{x,y\in [0,1] }\log\frac{\phi'_{\omega_1\cdots \omega_n}(x)}{\phi'_{\omega_1\cdots \omega_n}(y)}.\]
We say $\Phi$ has {\it mild distortion} if 
$D_1(\Phi)<\infty$ and  $D_n(\Phi)=o(n)$. 


\begin{prop}\label{sat-prop}
Let $\Phi=\{\phi_i\}_{i\in \mathcal I}$ be a 
parabolic finite IFS such that $\phi_i$ is $C^2$ for each $i\in \mathcal I$.
Then $\Phi$ is saturated and has mild distortion.
\end{prop}
\begin{proof}Using results in \cite{JT}, it was proved in \cite{T25} that if 
$\Xi=\{\xi_i\}_{i\in \mathcal J}$ is a parabolic infinite IFS such that $\xi_i$ is $C^2$ for each $i\in \mathcal J$ and
$\sup_{i\in \mathcal J}\max_{x\in[0,1]}|(\log| \xi'_{i}(x)|)'|$ is finite, then $\Xi$ is saturated and has mild distortion. Hence $\Phi$ has mild distortion. Since $\Phi$ is a finite IFS, the uniform boundedness on logarithms of derivatives clearly holds. Tracing the proof of \cite[Proposition~4.1]{T25} one can show that $\Phi$ is saturated. \end{proof}


\subsection{Existence of finite IFSs with large limit sets}\label{large-sec}
Given an IFS $\Phi=\{\phi_i\}_{i\in \mathcal I}$ on $[0,1]$, for $n\in\mathbb N$ and $\omega=(\omega_1,\ldots,\omega_n)\in {\mathcal I}^{\mathbb N}$ define {\it a fundamental interval of order $n$} by
\[I(\omega)=I(\omega_1,\ldots,\omega_n)=\phi_{\omega_1\cdots \omega_n}([0,1]).\]
For convenience, let us call $[0,1]$ the fundamental interval of order $0$.
For each $n\in\mathbb N$, $n$-th fundamental intervals are either disjoint, coincide or intersect only at their boundary points.

Given a parabolic IFS, using
the next proposition one can find a family of finite IFSs without parabolic indices whose limit sets approximate that of the original parabolic IFS in terms of Hausdorff dimension.

\begin{prop}[{\cite[Proposition~2.4]{T25}}]\label{katok} Let $\Phi=\{\phi_i\}_{i\in \mathcal I}$ be a parabolic IFS on $[0,1]$ that is saturated and has mild distortion.
For any $\varepsilon>0$ there exist 
an integer
 $p\geq2$, a constant $\gamma>0$ and 
a non-empty finite set $W^{(p)}\subset {\mathcal I}^{p}$ with the following properties:

\begin{itemize}

\item[(a)] for all distinct elements $\omega,\eta\in W^{(p)}$, $I(\omega)\cap I(\eta)=\emptyset;$

\item[(b)] for any $\omega=\omega_1\cdots \omega_{p}\in  W^{(p)}$, $\omega_p\in \mathcal I$ is not a parabolic index;

\item[(c)] for any $\omega\in W^{(p)}$, $\max_{x\in[0,1]}|\phi'_{\omega}(x)|<e^{-\gamma p}$;


\item[(d)]
the finite IFS
$\Phi^{(p)}=\{\phi_{\omega} \}_{\omega\in 
W^{(p)} }$ on $[0,1]$ has no parabolic index, and satisfies
\[\Lambda(\Phi^{(p)})\subset \Lambda(\Phi)
\text{ and }
\dim_{\rm H}\Lambda(\Phi^{(p) })>\dim_{\rm H}\Lambda(\Phi)-\varepsilon.\]
\end{itemize}
\end{prop}
\begin{rem}Although the condition 
$\Lambda(\Phi^{(p)})\subset \Lambda(\Phi)$ is not stated in \cite[Proposition~2.4]{T25}, it immediately follows from the proof there.\end{rem}



\subsection{Bounded distortion and bounds on  weak contraction}\label{bound-lem}
Given a parabolic IFS,
we will evaluate diameters of its fundamental intervals in terms of derivatives of the IFS.
To this end we need the following distortion bounds and lower bounds on contraction near 
 neutral fixed points.
\begin{lem}\label{dist-loc}
Let $\Phi=\{\phi_i\}_{i\in \mathcal I}$ be a 
parabolic IFS on $[0,1]$.
There exist constants $K_0\geq1$, $K_1>0$ such that for any parabolic index
$(j,x_j)$ of $\Phi$ with $x_j=0$ the following statements hold:
\begin{itemize}
\item[(a)] 
for all $x,y\in[\phi_{j}(1),1]$ and all $n\in\mathbb N$, 
\[\frac{|\phi'_{j^{n}}(y)|}{|\phi'_{j^{n}}(x)|}\leq K_0|x-y|;\] 
\item[(b)] 
for all $i\in \mathcal I\setminus\{j\}$, 
 $x\in[0,1]$ and all integer $n\geq2$, 
\[\frac{|\phi'_{j^{n}}(\phi_i(x))|}{|\phi_{j^n}(1)-\phi_{j^{n-1}}(1)|}\geq K_1.\]
\end{itemize}
\end{lem}

\if0\begin{lem}[in the proof of {\cite[Lemma~5.3]{JT}}]
\label{scope'}
Let $f\colon[0,1)\to \mathbb R$ be a $C^{2}$ map satisfying
$f(0)=0$, $f'(0)=1$ and $f'(x)>1$ for all $x\in(0,1)$. There exists a constant $C>0$ such that
for every $n\in\mathbb N$ and all $x,y\in J_{n-1}$,
\[\log\frac{(f^n)'(x)}{(f^n)'(y)}\leq 
C|f^n(x)-f^n(y)|\sum_{i=0}^{n-1}\frac{|J_{i}|}{|J_0|},\]
where 
$q_0=1$, $f(q_{i+1})=q_{i}$ and $J_i=[q_{i+1},q_i)$ for $i=0,\ldots, n-1$.
\end{lem}
\fi

\begin{proof}Since the number of parabolic indices is finite, (a) follows from 
a standard bounded distortion estimate for iterations of a $C^2$ map near its neutral fixed point, see e.g.,
the proof of \cite[Lemma~5.3]{JT}.

By the mean value theorem, for any  $n\in\mathbb N_{\geq2}$ there exists $y\in [\phi_{j^n}(1),\phi_{j^{n-1}}(1)]$ such that
\begin{equation}\label{dist-loc-eq}|(\phi_{j^n}^{-1})'(y)||\phi_{j^n}(1)-\phi_{j^{n-1}}(1)|=1-\phi_{j}(1)>0.\end{equation}
We have
$\phi_{j^n}^{-1}(y)\in[\phi_{j}(1),1]$,
and 
$\phi_{i}(x)\in[\phi_{j}(1),1]$
for all $i\in \mathcal I\setminus\{j\}$ and all $x\in[0,1]$.
By (a) and 
\eqref{dist-loc-eq}, we have 
\[|\phi'_{j^{n}}(\phi_i(x))|\geq e^{-K_0}
|\phi'_{j^{n}}(\phi_{j^n}^{-1}(y))|=\frac{e^{-K_0}}{1-\phi_{j}(1)}|\phi_{j^n}(1)-\phi_{j^{n-1}}(1)|.\]
Put $K_1=\inf_j( e^{K_0}/(1-\phi_{j}(1)))$ where the infimum is taken over all parabolic indices of $\Phi$.
Then (b) holds.
\end{proof}

\section{On the proofs of the main results}
In this section we complete the proofs of the main results.
In Section~\ref{BCF-IFS} we begin by introducing an infinite parabolic IFS generating the BCF.
In Section~\ref{pfthma} we complete the proof of Theorem~\ref{thma}. In Section~\ref{pfthmb} we complete the proof of Theorem~\ref{thmb}.

\subsection{Parabolic IFS generating the BCF}\label{BCF-IFS}
Consider the IFS $\Psi=\{\psi_i\}_{i\in \mathbb N_{\geq2}}$ on $[0,1]$ given by \[\psi_i(x)=1-\frac{1}{x+i-1}.\] 
We claim that $\Psi$ is a parabolic IFS.
Indeed, $2\in\mathbb N_{\geq2}$ is the only parabolic index: $\psi_2(0)=0$ and $\psi'_2(0)=1$. 
We have
\[\Lambda(\Psi)=\Lambda'(\Psi)=\{0\}\cup((0,1)\setminus\mathbb Q).\]
 Since $0=\lim_{n\to\infty}\psi_2^n(0)$,
 $0$ is contained in the limit set $\Lambda(\Psi)$.
A direct calculation shows (A1). Condition (A2) follows from (A1), Lemma~\ref{dist-loc}(a) and the finiteness of 
$\sup_{i\geq2}\max_{x\in[0,1]}|(\log| \psi'_{i}(x)|)'|$. Hence the claim holds.

For all $x\in \Lambda(\Psi)$ we have
 \[x=\lim_{n\to\infty}\psi_{b_1(x)}\cdots\psi_{b_n(x)}(0)=1-\confrac{1 }{b_{1}(x)} - \confrac{1 }{b_{2}(x)} -\confrac{1 }{b_{3}(x)} -\cdots.\] 
 In other words,
 a fundamental interval $I(b_1,\ldots,b_n)$ of order $n$ for the IFS $\Psi$ is the closure of the set of $x\in[0,1]$ that have the finite or infinite BCF expansion beginning $b_1,\ldots,b_n$.



 \subsection{Proof of Theorem~\ref{thma}}\label{pfthma}
Let $\mathcal B$ be a finite subset of $\mathbb N_{\geq2}$
with $\#\mathcal B\geq2$ and $2\in \mathcal B$. 
 It suffices to show that for any $\varepsilon>0$
there exists $\alpha\in(2,\infty)$ such that \begin{equation}\label{target}\dim_{\rm H}(G(\alpha)\cap F_{\mathcal B})>\frac{1}{1+\varepsilon}\left(\dim_{\rm H}F_{\mathcal B}-\varepsilon\right).\end{equation}
As explained in the introduction, a proof of this inequality consists of three steps.\medskip

\noindent{\it Step~1: Construction of a seed set.}
Consider the finite IFS
$\Psi_{\mathcal B}=\{\psi_i\}_{i\in \mathcal B }$
that is a subsystem of $\Psi$ in Section~\ref{BCF-IFS}. Since $2\in\mathcal B$, $\Psi_{\mathcal B}$ is a parabolic IFS and satisfies
$\Lambda(\Psi_{\mathcal B})=\{0\}\cup F_{\mathcal B}.$ By Proposition~\ref{sat-prop},
$\Psi_{\mathcal B}$ is saturated and has mild distortion. Let $\varepsilon>0$. 
By Proposition~\ref{katok} applied to  $\Psi_{\mathcal B}$,
 there exist an integer $p\geq2$, a constant $\gamma>0$ and 
a non-empty finite set $W^{(p)}\subset {\mathcal B}^{p}$ with the following properties:

\begin{itemize}

\item[(A)] for all distinct elements $\omega,\eta\in W^{(p)}$, $I(\omega)\cap I(\eta)=\emptyset;$

\item[(B)] for any $\omega=\omega_1\cdots \omega_{p}\in  W^{(p)}$, $\omega_p\in \mathcal B$ is not a parabolic index;

\item[(C)] for any $\omega\in W^{(p)}$, $\max_{x\in[0,1]}|\psi'_{\omega}(x)|<\exp(-\gamma p)$;



\item[(D)]
the finite IFS
$\Psi_{\mathcal B}^{(p)}=\{\psi_{\omega} \}_{\omega\in 
W^{(p)} }$ on $[0,1]$ has no parabolic indices and satisfies
\begin{equation}\label{include}\Lambda(\Psi_{\mathcal B}^{(p)})\subset F_{\mathcal B}\subset F_{\mathcal B}\cup\{0\}=\Lambda(\Psi_{\mathcal B})\end{equation} and 
\begin{equation}\label{dim-ineq}\dim_{\rm H}\Lambda(\Psi_{\mathcal B }^{(p)})>\dim_{\rm H}F_{\mathcal B}-\varepsilon.\end{equation}
\end{itemize}
The set $\Lambda(\Psi_{\mathcal B }^{(p)})$ is our seed set.\medskip

\noindent{\it Step~2: Modification of the seed set.}
We construct a subset of $G(\alpha)\cap F_{\mathcal B}$ by modifying the seed set $\Lambda(\Psi_{\mathcal B}^{(p)})$.
Fix \[t\in \mathcal B\setminus\{2\}.\]
Fix a sequence
$(m(i))_{i=0}^\infty$ of positive integers such that
\begin{equation}\label{M-eq}M_k\geq\max\left\{\sum_{i=0}^{k-1}M_i,\frac{M_{k-1}}{2}\right\}\ \text{ for all }k\geq1,\end{equation}
where
\begin{equation}\label{Mk-def}M_k=p\sum_{i=0}^k m(i).\end{equation}
 We put 
 \[\lambda=\frac{\varepsilon\gamma}{5}\]
 and
\[L=\max\{\max \mathcal B-2,p\},\]
 and fix $\alpha\in(2,\infty)$ such that
\begin{equation}\label{eq-c}2(\alpha-2)\log(\sqrt{2}(L+1))+(\alpha-2)\lambda\leq \frac{\lambda}{2}.\end{equation}
Let $y\in \Lambda(\Psi_{\mathcal B}^{(p)})$.
Then $(b_n(y))_{n\in\mathbb N}$ is a concatenation of elements of $W^{(p)}$.
At each position $M_k$, $k\geq0$
of $(b_n(y))_{n\in\mathbb N}$
we insert a block of $2$ of length $\lfloor\exp(\lambda M_{k})\rfloor$ and a single $t$ to define a new sequence 
\begin{equation}\label{ins}\begin{split}&\ldots,b_{M_{k-1}-1}(y),b_{M_{k-1} }(y),\!\!\!\!\!\!\underbrace{2,\ldots,2,}_{
 \lfloor\exp(\lambda M_{k-1} )\rfloor
 \text{-times}}\!\!\!\!\!\!t,b_{M_{k-1}+1}(y),\ldots\\
  &\ldots,b_{M_k-1}(y),b_{M_k}(y),\!\!\!\!\underbrace{2,\ldots,2,}_{
  \lfloor\exp(\lambda M_k)\rfloor\text{-times}}\!\!\!\!t,b_{M_k+1}(y),\ldots\\
   &\ldots,b_{M_{k+1} }(y),\!\!\!\!\!\underbrace{2,\ldots,2,}_{
  \lfloor\exp(\lambda M_{k+1})\rfloor\text{-times}}\!\!\!\!\!t,b_{M_{k+1}+1}(y),\ldots\end{split}\end{equation}
  Let $x(y)$ denote the point in $(0,1)\setminus\mathbb Q$ whose BCF expansion is given by this new sequence. 
  Since $(b_n(y))_{n\in\mathbb N}$ may contain $2$,  the inserted blocks of $2$ may not exhaust all the $2$ in $(b_n(x(y)))_{n\in\mathbb N}$.
Let $G^{(p)}(\alpha)$ denote the collection of these points:
\[G^{(p)}(\alpha)=\{x(y)\in (0,1)\setminus\mathbb Q\colon y\in \Lambda(\Psi_{\mathcal B}^{(p)})\}.\]

\begin{lem} \label{subset-lem}
We have $G^{(p)}(\alpha)\subset G(\alpha)\cap F_{\mathcal B}$.
\end{lem}
\begin{proof}By \eqref{include} and $2\in\mathcal B$,  $G^{(p)}(\alpha)\subset F_{\mathcal B}$ holds.
Let $y\in \Lambda(\Psi_{\mathcal B}^{(p)})$.
By the transformation formula from the BCF to RCF in Section~\ref{trans-sec}, 
for each $k\in\mathbb N$
the block of $2$ of length 
$\lfloor\exp(\lambda M_k)\rfloor$ in the BCF expansion of $x(y)$ is transformed into the single RCF partial quotient of $x(y)$ at a position, say $n+1$. 

Recall that the sequence of BCF partial quotients of $y$ is a concatenation of elements of $W^{(p)}$. 
By (B), the length of a block of $2$ contained in each element of $W^{(p)}$ does not exceed $p-1$. Hence,
each element of $W^{(p)}$ in the BCF expansion of $x(y)$ inherited from $y$
is transformed into a word from $\{1,\ldots,L\}$ of length not exceeding $p$.
Since $t\in\mathcal B\setminus\{2\}$, 
 if $1\leq i\leq n$ and $a_i(x(y))$ does not correspond to the inserted long block of $2$, then $a_i(x(y))\leq L$. 
Taking contributions from the inserted long blocks of $2$ together with this inequality, and using
 Lemma~\ref{qn-up},
\begin{equation}\label{subset-eq2}q_n(x(y))\leq (L+1)^n\prod_{i=0}^{k-1}\left(\lfloor\exp(\lambda M_i)\rfloor+1\right)\leq (L+1)^n2^k\exp\left(\lambda\sum_{i=0}^{k-1} M_i\right).\end{equation}
Since $k\leq M_k$ we have
\begin{equation}\label{subset-eq1}n+1\leq M_k+k\leq 2M_k+1.\end{equation}
Using \eqref{subset-eq2}, \eqref{subset-eq1}, $k\leq M_k$ and the first inequality in \eqref{M-eq}, \eqref{eq-c} we have
\[\begin{split}q_n^{\alpha-2}(x(y))&\leq (L+1)^{(\alpha-2)n}2^{(\alpha-2)k}\exp\left((\alpha-2)\lambda\sum_{i=0}^{k-1} M_i\right)\\
&\leq (\sqrt{2}(L+1))^{2(\alpha-2)M_k}\exp\left((\alpha-2)\lambda M_k\right)\\
&\leq \exp\left(\frac{\lambda M_k}{2}\right)\\
&<\lfloor\exp(\lambda M_k)\rfloor=a_{n+1}(x(y)).\end{split}\]
The last inequality holds for all sufficiently large $k$.
Lemma~\ref{identify} gives $x(y)\in G(\alpha)$. Since
$y\in \Lambda(\Psi_{\mathcal B}^{(p)})$ is arbitrary we obtain $G^{(p)}(\alpha)\subset G(\alpha)$.
\end{proof}

 \begin{rem}\label{effect-rem}
We have chosen $(m(i))_{i=0}^\infty$, $\lambda$, $\alpha$ in such a way that the effect of the block of $2$ inserted at position $M_k$ as in \eqref{ins} is `almost absorbed' into the contribution from the partial quotients at positions $M_{k-1}+1,\ldots,M_k$. For details, see the proof of Lemma~\ref{compare} in Step~3.\end{rem}

\noindent{\it Step~3: Estimate of the Hausdorff dimension of the modified set.}
The map $y\in \Lambda(\Psi_{\mathcal B}^{(p)})\mapsto x(y)\in G^{(p)}(\alpha)$ is bijective.
Let $f_\varepsilon\colon G^{(p)}(\alpha)\to \Lambda(\Psi_{\mathcal B}^{(p)})$ denote the inverse of this map, which eliminates all the inserted blocks of $2$ and all the inserted single $t$.

\begin{prop}\label{lip-prop}
 There exists $K>0$ such that
\[|f_\varepsilon(x)-f_\varepsilon(y)|\le K|x-y|^{\frac{1}{1+\varepsilon}}\ \text{ for all } x, y\in G^{(p)}(\alpha).\]\end{prop}
We finish the proof of Theorem~\ref{thma} assuming Proposition~\ref{lip-prop}. Combining Lemma~\ref{subset-lem}, Proposition~\ref{lip-prop}, \cite[Proposition~3.3]{Fal14} and \eqref{dim-ineq} we obtain
\[\dim_{\rm H}(G(\alpha)\cap F_{\mathcal B})\geq\dim_{\rm H} G^{(p)}(\alpha)\geq\frac{1}{1+\varepsilon}\dim_{\rm H}\Lambda(\Psi_{\mathcal B }^{(p)})>\frac{1}{1+\varepsilon}\left(\dim_{\rm H}F_{\mathcal B}-\varepsilon\right),\]
which verifies
\eqref{target}.

It is left to prove Proposition~\ref{lip-prop}.
Define a sequence $(n(k))_{k=0}^\infty$ of non-negative integers by $n(0)=0$, and
 \begin{equation}\label{n-ineq}n(k)=n(k-1)+m(k-1)p+\lfloor \exp(\lambda M_{k-1} )\rfloor+1\end{equation}
 for $k\geq1$. For each $k\geq1$,
  the position $n(k)$ in the sequence of BCF partial quotients of $x(y)$ is right after the end of 
the block of $2$ of length 
 $\lfloor \exp(\lambda M_{k-1} )\rfloor$. For each $k\in\mathbb N\cup\{0\}$,
 let $A_{k}$ denote the set of $b_1\cdots b_{n(k+1)}\in {\mathcal B}^{n(k+1)}$ for which the following conditions hold for all  $0\leq j\leq k$:
\begin{itemize}
\item[(I)] $b_{n(j)+1}\cdots b_{n(j)+m(j)p}\in {\mathcal B}^{m(j)p}$ is a concatenation of elements of $W^{(p)}$;
\item[(II)] $b_i=2$ for $n(j)+m(j)p+1\leq i< n(j+1)$;
\item[(III)] $b_{n(j+1)}=t$.
\end{itemize}
Notice that
\[G^{(p)}(\alpha)=\bigcap_{k=0}^\infty\bigcap_{\omega\in A_k}\psi_{\omega}([0,1]).\]
We set
\[B_k=\{
n(k)+mp\colon m=0,1,\ldots,m(k)\}.\]
Notice that \eqref{n-ineq} implies $\max B_k<n(k+1)$.
\begin{lem}\label{not2}
If $x\in G^{(p)}(\alpha)$, $k\in\mathbb N$ and $n\in B_k$,
then $b_n(x)\neq2$.
\end{lem}
\begin{proof}If $n\in B_k$ and $n\neq n(k)$, then (B) and (I) together imply $b_n(x)\neq2$. Condition (III) gives $b_{n(k)}(x)=t\neq2$.\end{proof}

For each $n\in \mathbb N$ and $b_1\cdots b_{n}\in {\mathcal B}^n$, let $\overline{b_1\cdots b_n}$ denote the word from ${\mathcal B}$ 
obtained by eliminating from $b_1\cdots b_n$ the digits $b_i$,
$n(j)+m(j)p+1\leq i\leq n(j+1)$ for $0\leq j\leq k$, where
$n(k)< n\leq n(k+1)$. 
Set
\[\overline{I}(b_1,\ldots,b_n)=\psi_{\overline{b_1\cdots b_n}}([0,1]).\]
Let $|\cdot|$ denote the Euclidean diameter of sets in $[0,1]$. 
\begin{lem}\label{compare}
There exists $\underbar{\it k}\in\mathbb N$ such that for any integer $k\geq \underbar{\it k}$,
 any $b_1\cdots b_{n(k+1)}\in A_{k}$ and any
$n\in B_k$
we have
\[|I(b_1,\ldots,b_n)|\geq|\overline{I}(b_1,\ldots,b_n)|^{1+\varepsilon}.\]
\end{lem}
\begin{proof}
 Since $I(b_1,\ldots,b_n)=\psi_{b_1\cdots b_n}([0,1])$ and $\overline{I}(b_1,\ldots,b_n)=\psi_{\overline{b_1\cdots b_n}}([0,1])$,
by the mean value theorem there exist $\theta,\overline{\theta}\in[0,1]$ such that  
\begin{equation}\label{un}|I(b_1,\ldots,b_n)|=|\psi_{b_1\cdots b_n}'(\theta)|\ \text{ and }\ |\overline{I}(b_1,\ldots,b_n)|=|\psi_{\overline{b_1\cdots b_n}}'(\overline{\theta})|.\end{equation}
Write
$n=n(k)+mp$, $0\leq m\leq m(k)$.
The chain rule gives
\begin{equation}\label{deux}\begin{split}|\psi_{b_1\cdots b_n}'(\theta)|=&
\prod_{i=0}^{k-1}|(\psi_{b_{n(i)+1}}\circ\cdots\circ\psi_{b_{n(i)+m(i)p} })'(\psi_{b_{n(i)+m(i)p+1}\cdots b_n }(\theta))|\\
&\times\prod_{i=0}^{k-1}|(\psi_{2}^{\lfloor \exp(\lambda M_{i})\rfloor})'(\psi_{b_{n(i+1)}\cdots b_n }(\theta))|
|\psi_t'(\psi_{b_{n(i+1)+1}\cdots b_n }(\theta))|\\
&\times
|\psi_{b_{n(k)}\cdots b_{n}}'(\theta)|\end{split}\end{equation}
and
\begin{equation}\label{trois}\begin{split}|\psi_{\overline{b_1\cdots b_n}}'(\overline\theta)|=&\prod_{i=0}^{k-1}|(\psi_{b_{n(i)+1}}\circ\cdots\circ\psi_{b_{n(i)+m(i)p }})'(\psi_{b_{n(i)+m(i)p+1}\cdots b_n }(\overline\theta))|\\
&\times
|\psi_{b_{n(k)}\cdots b_{n}}'(\overline\theta)|.\end{split}\end{equation}
Combining \eqref{un}, \eqref{deux}, \eqref{trois} we have
\[\begin{split}\frac{|\overline{I}(b_1,\ldots,b_n)|^{1+\varepsilon}}{|I(b_1,\ldots,b_n)| }=&\prod_{i=0}^{k-1}\frac{|(\psi_{b_{n(i)+1}}\circ\cdots\circ\psi_{b_{n(i)+m(i)p} })'(\psi_{b_{n(i)+m(i)p+1}\cdots b_n }(\overline\theta))|}{|(\psi_{b_{n(i)+1}}\circ\cdots\circ\psi_{b_{n(i)+m(i)p} })'(\psi_{b_{n(i)+m(i)p+1}\cdots b_n }(\theta))|}\\
&\times\prod_{i=0}^{k-1}\frac{|(\psi_{b_{n(i)+1}}\circ\cdots\circ\psi_{b_{n(i)+m(i)p} })'(\psi_{b_{n(i)+m(i)p+1}\cdots b_n }(\overline\theta))|^{\varepsilon}}{|(\psi_{2}^{\lfloor \exp(\lambda M_{i})\rfloor})'(\psi_{b_{n(i+1)}\cdots b_n }(\theta))|
|\psi_t'(\psi_{b_{n(i+1)+1}\cdots b_n }(\theta))|}\\
&\times\frac{|\psi_{b_{n(k)}\cdots b_{n}}'(\overline\theta)|}{|\psi_{b_{n(k)}\cdots b_{n}}'(\theta)|}|\psi_{b_{n(k)}\cdots b_{n}}'(\overline\theta)|^{\varepsilon}.\end{split}\]
In what follows
we estimate the fractions in the right-hand side from above one by one, and put together the estimates at the end.

By (A1) we have
\begin{equation}\label{back-eq1}|\psi_{b_{n(k)}\cdots b_{n}}'(\overline\theta)|^{\varepsilon}\leq1.\end{equation} 
Since
$b_{n(i)+m(i)p}\neq2$ 
 for $0\leq i\leq k-1$ and $b_n\neq 2$ by Lemma~\ref{not2}, (A2) gives
\begin{equation}\label{back-eq3}\frac{|(\psi_{b_{n(i)+1}}\circ\cdots\circ\psi_{b_{n(i)+m(i)p} })'(\psi_{b_{n(i)+m(i)p+1}\cdots b_n }(\overline\theta))|}{|(\psi_{b_{n(i)+1}}\circ\cdots\circ\psi_{b_{n(i)+m(i)p} })'(\psi_{b_{n(i)+m(i)p+1}\cdots b_n }(\theta))|}\leq C\end{equation}
for $0\leq i\leq k-1$ and
\begin{equation}\label{back-eq4}
\frac{|\psi_{b_{n(k)}\cdots b_{n}}'(\overline\theta)|}{|\psi_{b_{n(k)}\cdots b_{n}}'(\theta)|}\leq C.\end{equation}

The fraction in the second product is most delicate to treat, see Remark~\ref{effect-rem}.
Lemma~\ref{not2} gives
$b_{n(i+1)}=t\neq2$ for $0\leq i\leq k-1$, and a direct calculation gives
 $\psi_2^n(1)=1/(n+1)$ for $n\in\mathbb N$. By Lemma~\ref{dist-loc}(b), there exists a uniform constant $K_1>0$ such that
\begin{equation}\label{exp-eq}|(\psi_{2}^{\lfloor \exp(\lambda M_i)\rfloor})'(\psi_{b_{n(i+1)}\cdots b_n }(\theta))|\geq\frac{K_1}{
\exp(2\lambda M_i)
}\end{equation}
for $0\leq i\leq k-1$.
By the uniform contraction in (C) and the second inequality in \eqref{M-eq}, 
\begin{equation}\label{cont-eq}\begin{split}&|(\psi_{b_{n(i)+1}}\circ\cdots\circ\psi_{b_{n(i)+m(i)p} })'(\psi_{b_{n(i)+m(i)p+1}\cdots b_n }(\overline\theta))|\\&\leq \exp(-(M_i-M_{i-1})\gamma)\leq \exp\left(-\frac{M_i\gamma}{2}\right).\end{split}\end{equation}
By  \eqref{exp-eq} and \eqref{cont-eq} we have
\begin{equation}\label{back-eq5}
\begin{split}&\frac{|(\psi_{b_{n(i)+1}}\circ\cdots\circ\psi_{b_{n(i)+m(i)p} })'(\psi_{b_{n(i)+m(i)p+1}\cdots b_n }(\overline\theta))|^{\varepsilon}}{|(\psi_{2}^{\lfloor \exp(\lambda M_i)\rfloor})'(\psi_{b_{n(i+1)}\cdots b_n }(\theta))|
|\psi_t'(\psi_{b_{n(i+1)+1}\cdots b_n }(\theta))|}\\
&\leq\frac{
\exp(2\lambda M_i)
}{K_1\inf_{x\in[0,1]}|\psi_{t}'(x)|}\exp\left(-\frac{M_i\varepsilon\gamma}{2}\right)\leq\frac{1}{C^2}\end{split}\end{equation}
for $0\leq i\leq k-1$.
The last inequality holds for all sufficiently large $i$ by virtue of the choice of $\lambda$.
Combining \eqref{back-eq1}, \eqref{back-eq3}, \eqref{back-eq4}, \eqref{back-eq5} yields the desired inequality for all sufficiently large $k$.
\end{proof}

Put
 \begin{equation}\label{K-def}K_2=\min_{\substack{\omega,\eta\in W^{(p)}\\ \omega\neq \eta}}\min\{|x-y|\colon x\in I(\omega),\ y\in I(\eta)\}.\end{equation}
Fundamental intervals are closed subintervals of $[0,1]$, and all fundamental intervals of order $p$ corresponding to words in $W^{(p)}$ are pairwise disjoint by (A). Hence $K_2>0$ holds.
For a pair $(x,y)$ of distinct points in $G^{(p)}(\alpha)$, let $s(x,y)$ denote the maximal integer $n\geq0$
for which there exists a fundamental interval of order $n$ that contains $x$ and $y$.

\begin{lem}\label{separate}
For any pair $(x,y)$ of distinct points in $G^{(p)}(\alpha)$ satisfying $s(x,y)\geq p$, there exist integers $k\geq0$, $m\geq0$ such that \[n(k)+mp\leq s(x,y),\ 
n(k)+mp\in B_k\]
 and 
\[|I(b_1(x),\ldots,b_{n(k)+mp}(x))| \leq CK_2^{-1}|x-y|.\]\end{lem}
\begin{proof}
 There exists $k\geq 0$
such that 
$n(k)\leq s(x,y)\leq n(k+1)-1.$
By the definition of $s(x,y)$ and $b_{n(k+1)}(x)=b_{n(k+1)}(y)=t$ by (III),
the second inequality is actually strict, namely
\[n(k)\leq s(x,y)<n(k+1)-1.\]
From the definition \eqref{n-ineq}, there exists $m\in\{0,1,\ldots,m(k)-1\}$ such that
\begin{equation}\label{ineq}n(k)+mp\leq s(x,y)< n(k)+(m+1)p.\end{equation}
Hence $n(k)+mp\in B_k$ holds.
Since $s(x,y)\geq p$ and $n(0)=0$, $k=0$ implies $m\geq1$.

By \eqref{ineq}, we have
 $b_i(x)=b_i(y)$ for  $1\leq i\leq n(k)+mp$ and 
 $b_i(x)\neq b_i(y)$ for some $n(k)+mp+1\leq i\leq n(k)+(m+1)p$. 
  Set $b_i=b_i(x)$ for $1\leq i\leq n(k)+mp$, 
    $\psi=\psi_{b_1\cdots b_{n(k)+mp+1}}$
    and
   \[\omega(z)=b_{n(k)+mp+1}(z)\cdots b_{n(k)+(m+1)p}(z) \] for  $z=x,y$.
We have
$\psi^{-1}(x)\in I(\omega(x))$,
$\psi^{-1}(y)\in I(\omega(y))$ and 
$\omega(x)\neq  \omega(y)$.
By (I) we have $\omega(x)\in W^{(p)}$ and $\omega(y)\in W^{(p)}$.
By (A2) and the definition of $K_2$ in \eqref{K-def}, we have
\[\begin{split}&\frac{|I(b_1,\ldots,b_{n(k)+mp})|}{|x-y|}\leq C\frac{|\psi^{-1}(I(b_1,\ldots,b_{n(k)+mp} ))|}{|\psi^{-1}(x)-\psi^{-1}(y)|}\leq CK_2^{-1},\end{split}\]
as required.
\end{proof}


\begin{proof}[Proof of Proposition~\ref{lip-prop}]We set
\[K_3=\min\{|x-y|\colon x,y\in G^{(p)}(\alpha),\ x\neq y,\ s(x,y)\leq n(\underbar{\it k})-1 \},\]
where $\underbar{\it k}$ is the positive integer in Lemma~\ref{compare}.
Since $K_2>0$ we have $K_3>0$.
Let $(x,y)$ be a pair of distinct points in
$G^{(p)}(\alpha)$.
If $s(x,y)\leq n(\underbar{\it k})-1$ then 
 $|x-y|\geq K_3,$ and so 
\begin{equation}\label{holder-ineq1}|f_\varepsilon(x)-f_\varepsilon(y)|\leq 1\leq K_3^{-1}|x-y|.\end{equation}
If $s(x,y)\geq n(\underbar{\it k})\geq p$, then
let $k\geq 0$, $m\geq0$ be the integers for which the conclusion of Lemma~\ref{separate} holds. Clearly we have $k\geq\underbar{\it k}$. Since $x,y\in I(b_1(x),\ldots,b_{n(k)+mp}(x))$ we have
$f_\varepsilon(x),f_\varepsilon(y)\in\overline{I}(b_1(x),\ldots,b_{n(k)+mp}(x))$, and thus
\begin{equation}\begin{split}\label{holder-ineq2}|f_\varepsilon(x)-f_\varepsilon(y)|&\leq|
\overline{I}(b_1(x),\ldots,b_{n(k)+mp}(x))|\\
&\leq  |I(b_1(x),\ldots,b_{n(k)+mp}(x))|^{\frac{1}{1+\varepsilon}}\leq \left(CK_2^{-1} |x-y|\right)^{\frac{1}{1+\varepsilon}}.\end{split}\end{equation} 
To deduce the second inequality we have used
Lemma~\ref{compare}.
From \eqref{holder-ineq1} and \eqref{holder-ineq2},
the desired inequality in Proposition~\ref{lip-prop} holds.
\end{proof}

\subsection{Proof of Theorem~\ref{thmb}}\label{pfthmb}
Let $\mathcal B$ be a finite subset of $\mathbb N_{\geq2}$
with $\#\mathcal B\geq2$ and $2\in \mathcal B$.
By \cite[Proposition~1.3]{JT} (see also \cite{Jar28} and  \cite[Proposition~2,\ Remark~1]{DajKra00}), for any $\varepsilon>0$ there exists an integer $M\geq3$ such that $\dim_{\rm H}F_{\{2,\ldots,M\}}>1-\varepsilon$. 
Combining this with Theorem~\ref{thma} yields
\[\begin{split}\lim_{\alpha\to2+0}&\dim_{\rm H}(G(\alpha)\cap\{x\in\mathbb R\colon (b_n(x))_{n=1}^\infty\ \text{\rm is bounded}\})\\
&\geq\lim_{\alpha\to2+0}\dim_{\rm H}(G(\alpha)\cap F_{\{2,\ldots,M\}}  )=\dim_{\rm H}F_{\{2,\ldots,M\}}>1-\varepsilon.\end{split}\] Decreasing $\varepsilon$ to $0$ we obtain the desired equality.\qed

\section{On irrationality exponents for the BCF}
The irrationality exponent of each irrational $x$ is given by \begin{equation}\label{formula}\mu(x)=2+ \limsup_{n\to\infty}\frac{\log a_{n+1}(x)}{\log q_n(x)},\end{equation} see \cite[Remark~2]{DuvShi20} and \cite[Theorem~1]{Son04} for example.
Duverney and Shiokawa \cite{DuvShi24} asked whether 
there is an analogous formula for $\mu(x)$ in terms of the BCF. Define two sequences $(r_n(x))_{n=-1}^\infty$, $(s_n(x))_{n=-1}^\infty$ of integers inductively by
\begin{equation}\label{pq-BCF}\begin{split}&r_{-1}(x)=1,\ r_0(x)=b_0(x),\ r_n(x)=b_n(x)r_{n-1}(x)-r_{n-2}(x),\\
&s_{-1}(x)=0,\ s_0(x)=1,\ s_n(x)=b_n(x)s_{n-1}(x)-s_{n-2}(x).\end{split}\end{equation}
Then $|r_n(x)|$, $s_n(x)$ are coprime for all $n\geq1$ and we have
\[\lim_{n\to\infty}\frac{r_n(x)}{s_n(x)}=x.\]
Duverney and Shiokawa \cite[Theorem~1.1]{DuvShi24} proved 
that if the lengths of blocks of $2$ in  $(b_n(x))_{n=1}^\infty$ is bounded, then analogously to \eqref{formula} we have
\begin{equation}\label{formula1}\mu(x)=2+\limsup_{n\to\infty}\frac{\log b_{n+1}(x)}{\log s_n(x)}.\end{equation}

Using \eqref{pq-BCF} 
 it is easy to see that $s_n(x)\to\infty$ for any irrational $x$. Hence, the equality \eqref{formula1} breaks down if $\mu(x)>2$ and $(b_n(x))_{n=1}^\infty$ is bounded.
In \cite[Section~3]{DuvShi24}, Duverney and Shiokawa constructed irrationals with this property, but did not consider the Hausdorff dimension of sets of such irrationals. From Corollary~\ref{corb} it follows that
\eqref{formula1} breaks down in a set of Hausdorff dimension $1$.
\begin{thm}We have
\[\dim_{\rm H}\left\{x\in\mathbb R\colon\mu(x)>2+\limsup_{n\to\infty}\frac{\log b_{n+1}(x)}{\log s_n(x)}\right\}=1.\]
\end{thm}

\subsection*{Acknowledgments} 
The author thanks Hiroaki Ito and Kota Saito  for fruitful discussions.
This research was supported by the JSPS KAKENHI 25K21999.

\begin{thebibliography}{10}

\bibitem{BRS}
V. Becher, J. Reimann and T.A. Slaman.    Irrationality exponent, Hausdorff dimension and effectivization. {\it Monatsh. Math.} {\bf 185} (2018), 167--188. 

\bibitem{BDV01}
V. Beresnevich, D. Dickinson and S. Velani. Sets of exact ‘logarithmic’ order in the
theory of diophantine approximation. {\it Math. Ann.} {\bf 321} (2001), 253–273.
  
  \bibitem{Bes34} A. S. Besicovitch. Sets of fractal dimensions (IV): on rational approximation to real numbers. {\it J. London Math. Soc.} {\bf 9} (1934), 126--131.
  
  
  \bibitem{Bug03} Y. Bugeaud.  Sets of exact approximation order by rational numbers. {\it Math. Ann.} {\bf 327} (2003), 171--190.

\bibitem{Bug08} Y. Bugeaud.
Diophantine approximation and Cantor sets. {\it Math. Ann.} {\bf 341} (2008), 677--684.
  

\bibitem{DajKra00} K. Dajani and C. Kraaikamp.
The mother of all continued fractions. {\it Colloquium Math.} {\bf 84/85} (2000), 109--123.
  
  \bibitem{DuvShi20} D. Duverney and I. Shiokawa. Irrationality exponents of numbers related with Cahen's constant. {\it Monatsh. Math.} {\bf 191} (2020), 53--76.

  \bibitem{DuvShi24} D. Duverney and I. Shiokawa. Irrationality exponents of semi-regular continued fractions. {\it Tokyo J. Math.} {\bf 47} (2024), 89--109.
  



 
  
  

  \bibitem{Fal94} K. Falconer.
  Sets with large intersection properties. {\it J. London Math. Soc.} {\bf 49} (1994), 267--280.

 \bibitem{Fal14}K. Falconer. Fractal geometry - Mathematical foundations and applications (Third edition). {\it Wiley} 2014.

  
 \bibitem{G41}  I. J. Good. The fractional dimensional theory of continued fractions,  {\it Math. Proc. Camb. Philos. Soc.} {\bf 37} (1941), 199--228. 

 \bibitem{Gut63} R. G\"uting. On Mahler's function $\theta_1$. {\it Michigan Math. J.} {\bf 10}  (1963), 161--179.

 \bibitem{HV98} R. Hill and S. L. Velani.
 The Jarn\'ik-Besicovitch theorem for geometrically finite Kleinian groups. 
 {\it Proc. London Math. Soc.} {\bf 77} (1998), 524--550. 
 



 \bibitem{JT} J. Jaerisch and H. Takahasi. 
Mixed multifractal spectra of Birkhoff averages for non-uniformly expanding one-dimensional Markov maps with countably many branches.
      {\it Adv. Math.} {\bf 385} (2021), 107778.

       \bibitem{Jar28} 
       V. Jarn\'ik.
 Zur metrischen theorie der diophantischen approximationen.
 {\it Pr\'ace Mat. Fiz.} {\bf 36}
 (1928), 91--106.
 
 \bibitem{Jar29} V. Jarn\'ik. Diophantische approximationen und Hausdorffsches mass. {\it Matematicheskii
 Sbornik}
  {\bf 36}  (1929), 371--382.





  
  \bibitem{Khi24} A. Khintchine. Einige
  S\"atze \"uber Kettenbr\"uche, mit Anwendungen auf die Theorie der diophantischen Approximationen. {\it Math. Ann.} {\bf 92} (1924), 115--125.

  \bibitem{Khi64} A. Y. Khinchin. Continued Fractions. {\it University of Chicago Press III. London} 1964.
  





\bibitem{Mah84}
K. Mahler. Some suggestions for further research. {\it Bull. Austral. Math. Soc.} {\bf 29}  (1984), 101--108.


\bibitem{MauUrb03} R. D. Mauldin and M. Urba\'nski.
Graph directed Markov systems: geometry and Dynamics of Limit Sets. Cambridge Tracts in Math.
 {\bf 148}, {\it Cambridge University Press, Cambridge} 2003.

 

\bibitem{R57} A. R\'enyi, On algorithms for the generation of real numbers. {\it Magyar Tud. Akad. Mat. Fiz. Oszt. K\"ozl.} {\bf 7} (1957), 265--293.

\bibitem{R55}
K. F. Roth. Rational approximations to algebraic numbers. {\it Mathematika} {\bf 2} (1955), 1--20.

\bibitem{Sch71} W. M. Schmidt.  Approximation
to algebraic numbers. {\it Enseignement Math.} {\bf 17} (1971), 187--253.



\bibitem{Son04} J. Sondow. Irrationality measures, irrationality bases, and a theorem of Jarn\'ik.  arXiv:math/0406300 

  

\bibitem{T25} H. Takahasi.
 The distribution of the largest digit for parabolic Iterated Function Systems of the interval, {\it J. Fractal Geometry}, to appear.

  
 
 \end{thebibliography}
\end{document}